\documentclass[12pt,a4paper]{article}

\usepackage{amsmath,amssymb,amsfonts,amsthm}
\usepackage{times,hyperref,color}
\usepackage{graphicx}

\newtheorem{theorem}{\bf Theorem}[section]

\newtheorem{cor}[theorem]{\bf Corollary}

\newtheorem{conjecture}{\bf Conjecture}
\newtheorem{problem}[theorem]{Problem}

\usepackage{lipsum}

\begin{document}

		\title{\Large On the total versions of 1-2-3-conjecture for graphs and hypergraphs}
		\author{\small \hspace*{-1cm} Akbar Davoodi$^1$\footnote{The research of the first author was supported by the Czech Science Foundation, grant number GA19-08740S.}
		 \ and \  Leila Maherani$^2$\footnote{The research of the second author was supported by a grant from IPM.}\\
		{\footnotesize \hspace*{-1cm}$^{1}$The Czech Academy of Sciences, Institute of Computer Science, Pod Vod\'{a}renskou v\v{e}\v{z}\'{\i} 2, 182 07 Prague, Czech Republic}\\
		{\footnotesize \hspace*{-1cm}$^{2}$School of Mathematics, Institute for Research in Fundamental Sciences (IPM), P.O.Box: 19395-5746, Tehran, Iran}\\	
		{\footnotesize \hspace*{-1cm}{\bf E-mails:} davoodi@cs.cas.cz, maherani@ipm.ir}}
	\date {}
	\maketitle
	
\begin{abstract}
{\footnotesize
In 2004, Karo\'nski, \L uczak and Thomason proposed $1$-$2$-$3$-conjecture: For every nice graph $G$  there is an edge weighting function $ w:E(G)\rightarrow\{1,2,3\} $ such that the induced vertex coloring is proper. 	
After that, the total versions of this conjecture were suggested in the literature and recently,  Kalkowski et al. have generalized this conjecture  to  hypergraphs. In this paper, some previously known results on the total versions are improved. Moreover, an affirmative answer is given to the conjecture  for some well-known families of hypergraphs like complete $n$-partite hypergraphs,  paths, cycles, theta hypergraphs and some geometric planes.  Also, these hypergraphs are characterized based on the corresponding parameter.}
\end{abstract}
MSC 2010: 05c15; 05c78; 05c65

\section{Introduction}

Through out the paper,  we consider  simple and undirected graphs. For a graph $G$, the notations $V(G)$ and $E(G)$ stand for the vertex set and the edge set of  $G$, respectively.  A graph is called \textit{nice}, if it contains no component isomorphic to $ K_2 $.
Let $ G $ be a nice graph and $ w:E(G)\rightarrow\{1,2,\ldots, k\} $ be an integer edge weighting function of $G$.
Karo\'nski,  \L uczak and Thomason introduced a vertex coloring of $ G $ obtained from $w$ as follows \cite{KLT}.
For every $ v\in V(G) $, let $ \sigma^{e}(v) $ denotes the  sum of the  weights of edges incident on $ v $, i.e.
\begin{equation*}
\sigma^{e}(v) :=\sum_{e\ni v}w(e).
\end{equation*}
Whenever  the induced vertex coloring is \textit{proper}, i.e. for every two adjacent vertices $ u $ and $ v $, $ \sigma^{e}(u)\neq\sigma^{e}(v), $ the corresponding edge weighting is called \textit{neighbor sum distinguishing}.
The minimum integer $ k $ for which $ G $ admits a neighbor sum distinguishing is denoted by $ \chi^{e}(G) $.
In $ 2008 $, Karo\'nski et al. proposed the following statement  known as \textit{1-2-3-conjecture}.
\begin{conjecture}\label{conj:123}{\em\cite{KLT}}
	For every nice graph $ G $, $  \chi^{e}(G)\leq 3 $.
\end{conjecture}

This conjecture  has been explored in two ways. The first is to improve general upper bounds on $ \chi^{e}(G) $. The latter is to confirm it for some well-known families of graphs.

The chromatic number of a graph is the smallest number of colors required to give a proper vertex coloring for the graph.
A graph is called $ k $-colorable, if its chromatic number is at most $ k $.
In \cite{KLT}, it is  shown that if $G$ is  $ k $-colorable,
where $k$ is odd, then $ \chi^{e}(G) \leq k$. Consequently, the conjecture is true for
$3$-colorable graphs. In the literature, the first  constant bound was obtained by
Addario-Berry et al. \cite{A2} who showed that $ \chi^{e}(G) \leq 30$ for every nice graph $ G $. This result was improved to $ \chi^{e}(G) \leq 16$ by Addario-Berry, Dalal, and Reed \cite{A1} and then to  $ \chi^{e}(G) \leq 13$ by Wang
and Yu \cite{WY}. Currently, the best upper bound is $5$ due to  Kalkowski,  Karo\'nski and  Pfender  \cite{KKP}.

Parameter $  \chi^{e}$ has been studied for several classes of graphs
including paths, cycles,  complete graphs,  bipartite graphs, theta graphs and so on \cite{CH, DO, LYZ}.  However,   few number of graph classes are characterized  by the invariant $  \chi^{e}$.
To see these results in details we refer the reader to the survey written by Seamone \cite{survey}.

The cycle of length six and the complete graph on three vertices are simple instances of nice graphs for which weights 1 and 2 do not
suffice to give a neighbor sum distinguishing.
Hence, the bound claimed in Conjecture \ref{conj:123} is best possible. On the other hand, the following theorem gives  some intuitions for the correctness of Conjecture \ref{conj:123}.\\
The random graph $G_{n, p}$ is the graph on $n$ vertices such that there is an edge between any two vertices   randomly and independently, with
probability $p$.
\begin{theorem}{\em\cite{A1}}\label{a.all}
	Assume that $G=G_{n,p}$ is a random graph for any constant $p \in (0,1)$.
	Asymptotically almost surely, $  \chi^{e}(G) \leq 2$.
\end{theorem}
Clearly, $ \chi^{e}(G) =1$, if and only if $G$ contains no adjacent vertices with the same degree. Therefore,
determining the exact value of $  \chi^{e}$ seems to be a more difficult problem  for graphs with more regular structures. This observation  motivates researchers  to focus on some regular graphs and hypergraphs.

In the sequel, we present  definitions of  parameters $ \chi^{ve}$ and $ \chi^{ven}$ which are considered as the total versions of $ \chi^{e} $, introduced  in \cite{PW} and \cite{B}, respectively.
Let $w:V(G)\cup E(G)\rightarrow\{1,2,\ldots, k\} $ be a total integer weighting function of $G$. For every vertex $v$, $ \sigma^{ve}(v) $ is defined by
\begin{equation*}
\sigma^{ve}(v) :=w(v)+ \sum_{e\ni v}w(e)= w(v)+\sigma^{e}(v).
\end{equation*}
If $ \sigma^{ve}$ is a proper vertex coloring, then  $w$ is called  a \textit{neighbor sum distinguishing total coloring} of $G$. The minimum integer $k$ for which there exists a neighbor
sum distinguishing total coloring is denoted by $ \chi^{ve}(G) $.
As a total version of 1-2-3-conjecture, the following statement was suggested by Przyby\l o and Wo\'{z}niak in \cite{PW}.

\begin{conjecture}\label{ve conj}{\em\cite{PW}}
	For every  graph $ G $, $  \chi^{ve}(G)\leq 2 $.
\end{conjecture}

Kalkowski in \cite{K} showed that $  \chi^{ve}(G)\leq 3 $ for every graph $G$. Note that, $ G $ is not necessarily a nice graph in Conjecture \ref{ve conj}.

Let $w:V(G)\cup E(G)\rightarrow\{1,2,\ldots, k\} $ be a total integer weighting function of $G$. For every vertex $v \in V(G)$,
\begin{equation*}
\sigma^{ven}(v) :=w(v)+ \sum_{e\ni v}w(e)+ \sum_{u\in N(v)}w(u)=\sigma^{ve}(v)+\sum_{u\in N(v)} w(u),
\end{equation*}
where $N(v)$ denotes the open neighborhood of $v$. The function $w$ is  a \textit{neighbor full sum distinguishing
	total k-coloring} of $ G, $ whenever $\sigma^{ven}$ is a proper vertex coloring. The minimum integer $k$ for which there exists
such a coloring is denoted by $  \chi^{ven}(G)$.
Recently, Baudon et al. have introduced this notation  and  have proposed  the following conjecture \cite{B}.
\begin{conjecture}\label{ven conj}{\em\cite{B}}
	For every  nice graph $ G $, $  \chi^{ven}(G)\leq 3 $.
\end{conjecture}
They also applied a similar proof technique stated in \cite{KKP} and showed  that $\chi^{ven}$ is bounded by $ 5 $ in general. Moreover, They gave an affirmative answer to Conjecture \ref{ven conj} for paths, cycles, bipartite graphs,  complete graphs and split graphs.

In section \ref{graph}, we  focus on graphs and present the main results of this section in   theoremS \ref{g1} and \ref{g2}.
Let $ K_n(t) $ be the complete $ n$-partite graph each part of size $ t $.  Davoodi and Omoomi in \cite{DO} showed that  $ \chi^{e}(K_n(t)) =2$ for $n,t \geq 2$.  Theorem \ref{g1} strengthens this result by determining  of $ \chi^{e}$ for more complete $ n$-partite graphs.
The authors in \cite{KLT} asserted that for every $3$-partite graph $G$, $ \chi^{e}(G)\leq 3$ and consequently, it is proved that $ \chi^{ven}(G)\leq 3$ \cite{B}.  In Theorem \ref{g2}, we determine the exact value of $\chi^{ven}$ for complete $3$-partite graphs.
\\

A {\it hypergraph} $\mathcal{H}$ is an ordered pair $\mathcal{H}=(V,E)$, where $V$ is a finite and non-empty set of vertices and $E$ is a collection of distinct non-empty  subsets of $V$.  The set of vertices and the set of edges of $\mathcal{H}$ are denoted by  $V(\mathcal{H})$ and $E(\mathcal{H})$, respectively.
An {\it $r$-uniform
	hypergraph} is a hypergraph such that all of its edges have size $r$. For a vertex  $v \in V(\cal{H})$, $N(v):=\{u: \exists\ e \in E(\mathcal{H}), \{v,u\}\subseteq e\}$ stands for the neighborhood of $v$.
A hypergraph vertex coloring is called \textit{proper} when each edge of $\mathcal{H}$ contains at least two vertices with
distinct colors.
Similar to the graph case,  $  \chi^{e}(\mathcal{H})$ is the minimum integer $ k $
for which there exists a weighting function $ w:E(\mathcal{H})\rightarrow\{1,2,\ldots, k\}$  inducing a proper vertex coloring. Also,  parameters $  \chi^{ve}$ and $  \chi^{ven}$ are defined for hypergraphs, in a similar manner.

Conjecture \ref{conj:123} has been generalized to $3$-uniform hypergraphs by Kalkowski,  Karon\'{s}ki, and Pfender in the following form \cite{KKP-H}.

\begin{conjecture}{\em\cite{KKP-H}}\label{conj:3-HYPER}
	For every $3$-uniform hypergraph $ \mathcal{H} $ with no  isolated edge, $ \chi^{e}(\mathcal{H})\leq 3 $.
\end{conjecture}
Let  $r\geq 3$ be an integer. In \cite{KKP-H}, it is shown that for any $r$-uniform hypergraph $ \cal H $ with no
isolated edge, $\chi^{e}(\mathcal{H})\leq{\max}\{5, r + 1\}$. Also, the authors in \cite{KKP-H} have proved that $\chi^{ve}(\mathcal{H})\leq 3$ for every hypergraph $\mathcal{H}$ with no edges of size $0$ or $1$.

The random hypergraph  $ \mathcal{H}^{r}_{n,p} $ is an $r$-uniform hypergraph on $n$ vertices such that each  edge is appeared randomly and independently, with probability $p$.
Applying probabilistic arguments, Bennett et al.  established  the following theorem  \cite{BDFL}.
\begin{theorem}\label{thm:almost}{\em\cite{BDFL}}
	Assume that $ \mathcal{H}= \mathcal{H}^{r}_{n,p}  $ is an $ r $-uniform random hypergraph. Asymptotically almost surely,  $ \chi^{e}(\mathcal{H})=2 $, if $ r=3 $ and   $ \chi^{e}(\mathcal{H})=1 $, whenever $ r\geq 4  $.
\end{theorem}

Moreover, they extended Conjecture \ref{conj:3-HYPER} to $r$-uniform hypergraphs as follows.
\begin{conjecture}\label{conj:r-HYPER}{\em\cite{BDFL}}
	Let $r \geq 3$ and  $\mathcal{H}$ be an $r$-uniform hypergraph  with no isolated edge, then $ \chi^{e}(\mathcal{H})\leq 3. $
\end{conjecture}
It is inferred from Theorem \ref{thm:almost} that every edge in an  $ r $-uniform random hypergraph with $ r\geq 4 $ contains two vertices of different degrees with   high probability. This fact motivates researchers to explore regular $ r $-uniform hypergrphs in Conjecture \ref{conj:r-HYPER} as the most challenging structures.

In section \ref{hgraphs}, we determine the exact values of  $\chi^{e}$,  $\chi^{ve}$ and  $\chi^{ven}$ for the well-known  families of hypergraphs.  Indeed, we will see that some  $ r $-uniform hypergraphs can not be weighted by 1 and 2 and hence, three colors are required. Moreover, we show that three colors are sufficient for them.
Briefly,  main results of section  \ref{hgraphs} are as follows. Theorem \ref{complete} determines  the parameter $ \chi^{e}$ for the complete $ n$-partite $ r$-uniform hypergraph with parts of equal sizes. This theorem is a generalization of the result in \cite{DO} to hypergraphs. Moreover, as a straightforward corollary, we get $ \chi^{e}(\mathcal{K}_n^r)=2 $ where $\mathcal{K}_n^r$ denotes the complete  $ r$-uniform hypergraph of order $n$.
Also, correctness of Conjecture \ref{conj:r-HYPER} is confirmed for $t$-tight paths and $t$-tight cycles in Theorems \ref{path} and \ref{cycle}, respectively. Consequently, the exact value of $\chi^{e}$ is determined for loose paths, loose cycles, tight paths and tight cycles.
The theta hypergraph $\mathcal{H}_{\Theta}$ is considered as a generalization of $t$-tight cycle and the value of $\chi^{e}(\mathcal{H}_{\Theta})$ is determined by Theorem \ref{teta}.
Another regular structure that we explore is the family of geometric planes.  In Theorem \ref{plane}, the parameter $\chi^{e}$ is determined  for affine planes and projective planes which are known as the most applicable geometric planes.

\section{Main results for graphs\label{graph}}
In this section, we will prove two theorems in graph case.
Davoodi and Omoomi proved that $ \chi^{e}(K_n(t)) =2$, for $n,t \geq 2$ \cite{DO}. In the first theorem,  we have eliminated the equality requirement on the parts. Precisely,  we demonstrate that $ \chi^{e}(G) =2$, where $G$ is  complete $n$-partite graph with parts of  different sizes.
The second  theorem  specifies  $\chi^{ven}$ for any complete $ 3$-partite graph $G$. This statement improves the result $ \chi^{ven}(G)\leq 3$ due to Baudon et al. in  \cite{B}, as well.

\begin{theorem}\label{g1}
	Let  $G$ be a complete $ n$-partite graph with at least $\lfloor \frac{n-1}{2}\rfloor$ parts of size greater than one. Then $ \chi^{e}(G) =2$.
\end{theorem}
\begin{proof}
	First, assume that $ n=2k $ for some $k\in \mathbb{N}$ and $X_1, X_2, \ldots ,X_k, Y_1,Y_2,\ldots ,Y_k$ are  parts of $G$ such that $|X_i|=m_i$, $|Y_i|=n_i$,  for $ 1\leq i\leq k $ and  $m_1\geq m_2\geq \cdots \geq m_k\geq n_k\geq n_{k-1}\geq \cdots \geq n_1$.
	Let $ B[X,Y] $ be the  graph with parts $ X=\{x_1,x_2,\ldots, x_k\} $ and $ Y=\{y_1,y_2,\ldots, y_k\} $ in which $ B[Y]\cong K_k $ and $ N_B(x_i)=\{y_1,y_2,\ldots, y_i\} $ for $ 1\leq i \leq k $.
	Suppose that $ B^* $ is the graph obtained from $ B $ by blowing up  every vertex $x_i$ with a set of  size $m_i$ and every vertex $y_i$ with a set of  size $n_i$, for $1\leq i \leq k $. Note that $B^*$ is a subgraph of $G$.
	We denote by $X_i $ ( resp. $Y_i$)   the set of  vertices corresponding to $x_i$ ( resp. $y_i$) in $V(B)$,  for every   $1\leq i \leq k $.\\
	
	\textbf{Case 1.} $m_k > n_k$.\\
	In this case we define $w(e)=2$ if $e\in E(B^*)$ and $w(e)=1$, otherwise.
	Hence, for every vertex $ v\in V(G) $ we have,
	\begin{center}
		\begin{equation}\label{1}
		\sigma^{e}(v)=
		\left\{
		\begin{array}{ll}
		|V(G)|+\sum _{j=1}^i n_j-m_i & \ \ \ v\in X_i,  \\
		|V(G)|+\sum _{j=i}^k m_j + \sum_{j=1}^k n_j-2n_i & \ \ \ v\in Y_i.
		\end{array}
		\right.
		\end{equation}
	\end{center}
	We claim that $ w $ induces a proper vertex coloring for $ G $. Let $1\leq \ell <\ell'  \leq k $.
	Suppose to the contrary that $\sigma^{e}(x_{\ell})=\sigma^{e}(x_{\ell'})$. Then, $n_{\ell +1}+\cdots +n_{\ell'}+m_{\ell}=m_{\ell'},$
	which is impossible since $m_{\ell}\geq m_{\ell'}$ and $n_{\ell +1}+\cdots +n_{\ell'} >0$. Now, suppose that $\sigma^{e}(y_{\ell})=\sigma^{e}(y_{\ell'})$. Thus,
	$m_{\ell}+\cdots +m_{\ell' -1}+2n_{\ell'}=2n_{\ell}.$
	It is impossible since $n_{\ell'}\geq n_{\ell}$ and $m_{\ell}+\cdots +m_{\ell' -1} >0$.
	If  $\sigma^{e}(x_{\ell})=\sigma^{e}(y_{\ell'})$. Then, $\sum_{j=1}^k n_j+\sum _{j=\ell'}^k m_j+m_{\ell}=\sum _{j=1}^{\ell} n_j +2n_{\ell'}.$
	It doesn't occur because $ \sum_{j=1}^k n_j >\sum _{j=1}^{\ell} n_j$ and  $m_k+m_{\ell} \geq 2n_{\ell'}$.
	Now, let $\sigma^{e}(x_{\ell'})=\sigma^{e}(y_{\ell})$. Hence, $\sum_{j=1}^k n_j+\sum _{j=\ell}^k m_j+m_{\ell'}=\sum _{j=1}^{\ell'} n_j +2n_{\ell}.$
	It is impossible since $ \sum_{j=1}^k n_j \geq \sum _{j=1}^{\ell'} n_j$ and $m_{\ell}+m_k+m_{\ell'} >2n_{\ell}$.
	If  $\sigma^{e}(x_{\ell})=\sigma^{e}(y_{\ell})$, for some  $1\leq \ell \leq k $. Then, $\sum _{j=1}^{\ell} n_j +2n_{\ell}=m_{\ell}+\sum _{j=\ell}^k m_j+\sum_{j=1}^k n_j$. It is impossible since $\sum_{j=1}^k n_j\geq \sum _{j=1}^{\ell} n_j $, $m_{\ell}+m_k \geq 2n_{\ell}$ and $m_k >n_k \geq n_{\ell}$.\\
	
	\textbf{Case 2.} $n_k = m_k$ and $m_{k-1}>m_k.$\\
	Consider a fixed vertex $z$ in $X_{k-1}$.  Now, define weighting function $w$ as follows. Set $w(e)=2$ if $e\in E(B^*)$ or $e=zv$ for every $v\in X_k$ and $w(e)=1$, otherwise. Thus, for every vertex $ v\in V(G) $ we have,
	\begin{center}
		\begin{equation}\label{1}
		\sigma^{e}(v)=
		\left\{
		\begin{array}{ll}
		|V(G)|+\sum _{j=1}^i n_j-m_i & \ \ \ v\in X_i, i\neq k, v\neq z,  \\
		|V(G)|+\sum _{j=1}^{k-1} n_j +m_{k}-m_{k-1} & \ \ \ v=z,\\
		|V(G)|+\sum _{j=1}^{k} n_j -m_{k}+1 & \ \ \ v\in X_k, \\
		|V(G)|+\sum _{j=i}^k m_j + \sum_{j=1}^k n_j-2n_i & \ \ \ v\in Y_i.
		\end{array}
		\right.
		\end{equation}
	\end{center}
	Similar to the previous case we show that $w$ induces a proper vertex coloring for $G$ in this case. Consider a vertex $u \in X_k$. Suppose that $\sigma^{e}(u)=\sigma^{e}(x)$, for some  $x\in X_{\ell}$, $1\leq \ell \leq k-1 $ and $x \neq z$. Therefore, $\sum _{j=1}^{k} n_j + m_{\ell}+1=m_k +\sum _{j=1}^{\ell} n_j$. This is a contradiction since $\sum _{j=1}^{k} n_j >\sum _{j=1}^{\ell} n_j$ and $m_{\ell} > m_k$.
	If $\sigma^{e}(u)=\sigma^{e}(y)$, for some  $y \in Y_{\ell}$, $1\leq \ell \leq k $, then $2n_{\ell}+1=m_k + \sum _{j=\ell}^k m_j.$
	If $\ell =k$, then  we have $2n_{\ell}+1=2m_k$ that does not hold. Otherwise $1\leq \ell \leq k-1$. Since $m_{\ell}> m_k \geq n_{\ell}$ and $m_{\ell} >1$,  it is clear that $2m_k +m_{\ell} >2n_{\ell}+1$ that is a contradiction.
	Now let $\sigma^{e}(u)=\sigma^{e}(z)$. Thus, $n_k +m_{k-1}+1=2m_{k}.$
	Equality $m_k=n_k$ implies $m_{k-1}+1=m_{k}$ that contradicts  $m_{k-1}> m_{k}$.
	If $\sigma^{e}(z)=\sigma^{e}(x)$ for some  $x\in X_{\ell}$, $1\leq \ell \leq k-2 $, then
	$\sum _{j=1}^{k-1} n_j +m_k+m_{\ell}=\sum _{j=1}^{\ell} n_j+m_{k-1}$. It does not occur since $\sum _{j=1}^{k-1} n_j >\sum _{j=1}^{\ell} n_j$ and $m_{\ell} \geq m_{k-1}$.
	Finally, let $\sigma^{e}(z)=\sigma^{e}(y)$ for some  $y\in Y_{\ell}$, $1\leq \ell \leq k $. Hence, $2n_{\ell}=m_{k-1}+\sum _{j=\ell}^k m_j $. It contradicts $m_{k-1} >m_k \geq n_{\ell}$.\\
	
	\textbf{Case 3.} $n_k = m_k=m_{k-1}>1$\\
	Let $X_{k}=\{u_1,u_2,\ldots ,u_p\}$, $X_{k-1}=\{u'_1,u'_2,\ldots ,u'_p\}$ and $A=\{u_1u'_1, u_2u'_2,\ldots ,u_pu'_{p}\}\subset E(G)$. Here, we present a weighting function as follows. We define  $w(e)=2$ if $e\in E(B^*)\cup A$ and $w(e)=1$ otherwise. Therefore, for every vertex $ v\in V(G) $ we have,
	\begin{center}
		\begin{equation}\label{1}
		\sigma^{e}(v)=
		\left\{
		\begin{array}{ll}
		| V(G)|+\sum _{j=1}^i n_j-m_i & \ \ \ v\in X_i, 1\leq i\leq k-2,  \\
		| V(G)|+\sum _{j=1}^{k-1} n_j -m_{k-1}+1 & \ \ \ v\in X_{k-1},\\
		|V(G)|+\sum _{j=1}^{k} n_j -m_{k}+1 & \ \ \ v\in X_k, \\
		|V(G)|+\sum _{j=i}^k m_j + \sum_{j=1}^k n_j-2n_i & \ \ \ v\in Y_i.
		\end{array}
		\right.
		\end{equation}
	\end{center}
	Consider two fixed vertices $u\in X_k$ and $u'\in X_{k-1}$. Suppose that $\sigma^{e}(u)=\sigma^{e}(x)$ for some  $x\in X_{\ell}$, $1\leq \ell \leq k-2. $ Then, $ m_{k}=\sum _{j={\ell}+1}^k n_j+m_{\ell}+1$ that is a contradiction. If $\sigma^{e}(u')=\sigma^{e}(x)$ for some  $x\in X_{\ell}$, $1\leq \ell \leq k-2 $, then $ m_{k-1}=m_{\ell}+\sum _{j={\ell}+1}^{k-1} n_j+1$. It is not possible since $m_{\ell}+1>m_{k-1}$. By the assumptions $m_{k}=m_{k-1}$ and $n_k >0$, we have  $\sigma^{e}(u)>\sigma^{e}(u')$. Now, let $\sigma^{e}(u)=\sigma^{e}(y)$ for some  $y\in Y_{\ell}$, $1\leq \ell \leq k. $ Hence, $2n_{\ell}+1=m_{k}+ \sum _{j={\ell}}^{k} m_j$. It does not occur since $m_{\ell}\geq m_k \geq n_{\ell}$ and $m_k >1$. If $\sigma^{e}(u')=\sigma^{e}(y)$ for some  $y\in Y_{\ell}$, $1\leq \ell \leq k. $ Then, $2n_{\ell}+1=m_{k-1}+\sum _{j={\ell}}^{k} m_j +n_k$ which is not possible.
	Since $u$ and $u'$ are chosen arbitrarily, we are done.
	
	Now, we may assume that $n=2k+1$ for some integer $k$. Let $V_1,V_2,\ldots ,V_{n}$ be the  parts of $G$ such that $|V_i|\geq |V_{i+1}|$ for every $1\leq i \leq n-1$. Therefore, $deg_G(v)\leq deg_G(u)$ for every $v\in V_1$ and $u \in V(G)\setminus V_1$ and  graph $G-V_1$ is a complete $(n-1)$-partite graph, where $n-1=2k$. Now, construct graph $B^*$ on $G-V_1$, as mentioned above. Clearly, $deg_{B^*}(v)=0$ for every vertex $v\in V_1$ and  $deg_{B^*}(v)>0$ for remaining vertices of $G$. Here, we consider weighting function $w$ similar to what applied in the previous cases. One can easily see that $\sigma^{e}(v)=deg_{G}(v)\leq deg_{G}(u)<\sigma^{e}(u)$ for every  $v\in V_1$ and $u\in V(G)\setminus V_1$. Also, according to what mentioned above, we have $\sigma^{e}(u)\neq \sigma^{e}(u')$ for every $u,u'\in V(G)\setminus V_1$.
\end{proof}

\begin{theorem}\label{g2}
	Let $ G \ncong K_3 $ be  the complete $ 3$-partite graph with parts $ V_1, V_2, V_3 $. Then
	\begin{equation*}
	\chi^{ven}(G)=
	\left\{
	\begin{array}{cl}
	1 & \ \ \ |V_i|\neq|V_j|,\  i\neq j,\  1\leq i,j\leq 3,  \\
	2 & \ \ \ otherwise.
	\end{array}
	\right.
	\end{equation*}
\end{theorem}
\begin{proof}
 	Obviously, if $ |V_i|\neq|V_j| $ for $ 1\leq i\neq j\leq 3 $, then  $ w : V(G)\cup E(G) \rightarrow \{1\}$ turns out a proper vertex coloring and consequently, $ \chi^{ven}(G)=1 $. For the case $ |V_1|=|V_2|=|V_3|=t $,  Theorem \ref{g1} implies that $\chi^{e}$ and consequently $ \chi^{ven} $ are equal to 2.
	Thus, one may assume that $ G $ contains exactly two parts $ V_1 $ and $ V_2$ of a same size. We consider the following two cases and in each case present  a total weighting function  which induces a proper vertex coloring.\\
	
	\textbf{Case 1.}
	$ |V_1|= |V_2| =t $ and $ |V_3|\notin\{t,\frac{2t}{3}\} $.\\
	Let $n=|V_1|+|V_2|+|V_3|$. Then $n>2t$. Define function $ w:V(G)\cup E(G)\rightarrow \{1,2\} $ as follows. For every vertex $ v\in V(G) $, $ w(v)=1 $ and
	\begin{equation*}
	w(e)=
	\left\{
	\begin{array}{ll}
	1 & \ \ \ e\in [V_2,V_3],  \\
	2 & \ \ \ \rm{otherwise,}
	\end{array}
	\right.
	\end{equation*}
	in which  notation $ e\in [V_2,V_3]$  means that edge $e$ is between parts $V_2$ and $V_3$.
	Now, let $ v_i\in V_i $, $ 1\leq i\leq 3 $. By definition of $\sigma^{ven}$, we have
	\begin{align*}
	\sigma^{ven}(v_1) &=1+\big(2t+2(n-2t)\big)+\big(t+(n-2t)\big)=3n-3t+1,\\
	\sigma^{ven}(v_2)&=1+\big(2t+(n-2t)\big)+\big(t+(n-2t)\big)=2n-t+1,\\
	\sigma^{ven}(v_3)
	&=1+3t+2t=5t+1.
	\end{align*}
	Note that the assumptions $ |V_1|= |V_2| =t $ and $ |V_3|\notin\{t,\frac{2t}{3}\} $ concludes that $ \big|\{3n-3t+1, 2n-t+1, 5t+1\}\big|=3 $. Hence, this coloring is proper.\\
	
	\textbf{Case 2.}
	$ |V_1|= |V_2| =t $ and $ |V_3|=\frac{2t}{3} $.\\
	In this case, let
	\begin{equation*}
	w(e)=
	\left\{
	\begin{array}{cl}
	2 & \ \ \ e\in [V_1,V_2],  \\
	1 & \ \ \ \text{otherwise,}
	\end{array}
	\right.
	\qquad and \qquad
	w(v)=
	\left\{
	\begin{array}{cl}
	2 & \ \ \ v\in V_1,  \\
	1 & \ \ \ \text{otherwise.}
	\end{array}
	\right.
	\end{equation*}
	For $ v_i\in V_i $, $ 1\leq i\leq 3 $,
	\begin{align*}
	\sigma^{ven}(v_1)&=2+(2t+\frac{2t}{3})+(t+\frac{2t}{3})=4t+\frac{t}{3}+2,\\
	\sigma^{ven}(v_2)&=1+(2t+\frac{2t}{3})+(2t+\frac{2t}{3})=5t+\frac{t}{3}+1,\\
	\sigma^{ven}(v_3)&=1+2t+3t=5t+1.
	\end{align*}
	Since $ t, \frac{2t}{3}\in \mathbb{N} $, every pair of adjacent vertices receive different colors.
\end{proof}

\section{Main results for hypergraphs}\label{hgraphs}
In this section, first some preliminary definitions will be given and then a number of results will be presented for hypergraphs.

Let $\mathcal{H}$ be a hypergraph.
For  a vertex $v \in V(\mathcal{H})$,\textit{ degree }of $v$ in $\mathcal{H}$, denoted by $deg_\mathcal{H}(v)$, is the number of edges in $\mathcal{H}$ containing $v$. We say that $\mathcal{H}'$ is a subhypergraph of $\mathcal{H}$, if $V(\mathcal{H}')\subseteq V(\mathcal{H})$ and $E(\mathcal{H}')\subseteq E(\mathcal{H})$. A subhypergraph $\mathcal{H}'$ of $\mathcal{H}$ is called spaning, whenever  $V(\mathcal{H}')= V(\mathcal{H})$.

\subsection{Complete hypergraphs}
An $r$-uniform hypergraph is called \textit{$n$-partite} if there is a partition of the vertex set into $n$ parts such that each edge
has at most  one vertex in each part.
A {\it complete $r$-uniform  hypergraph} of
order $n$, denoted by { $\mathcal{K}_n^r$}, is a hypergraph  consisting of
all the $r$-subsets of  vertex set $V$ of cardinality $n$.

\begin{theorem}\label{complete}
	Let $ r\geq 3 $ and $\mathcal{K}_n^r(t) $ be the complete $ r$-uniform $ n$-partite  hypergraph on vertex sets  $ V_1,\ldots, V_n $ such that  $\vert V_i\vert = t$, for  $1\leq i\leq n$. If $ n>2(r-1)^2 $, then $ \chi^{e}(\mathcal{K}_n^r(t))=2 $.
\end{theorem}

\begin{proof}	
	Let $\mathcal{H}=\mathcal{K}_n^r(t)$. We specify a subhypergraph $ \mathcal{H}^*\subset \mathcal{H} $ such that assigning weight $ 2 $ to the edges of $ \mathcal{H}^* $ and weight $ 1 $ to the remaining edges of $ \mathcal{H} $  yields a proper vertex coloring for $ \mathcal{H} $.
	In the other words, every arbitrary edge in $\mathcal{H} $ has two vertices with different colors. Since $ \mathcal{H} $ is regular and for every vertex $ u \in V(\mathcal{H}) $ we have, $$ \sigma ^{e}(u)=\sum_{e\ni u}w(e)=2deg_{\mathcal{H}^*}(u)+deg_{\mathcal{H}\backslash \mathcal{H}^*}(u)=deg_{\mathcal{H}}(u)+deg_{\mathcal{H}^*}(u), $$
	it suffices to demonstrate  that every edge in $ \mathcal{H} $ contains two vertices $ u $ and $ u' $ such that
	$ deg_{\mathcal{H}^*}(u)\neq deg_{\mathcal{H}^*}(u') $.
	Suppose that $ n=(r-1)p+q $, where $ 0\leq q\leq r-2 $. Now, partition vertex sets $ V_1,\ldots, V_n $  into  $ p $  disjoint classes $ \mathcal{A}_1, \mathcal{A}_2\ldots, \mathcal{A}_p $ each of  size $ r-1 $ and a possible class $ \mathcal{A}_0 $ of size $ q $, i.e.
	\begin{align*}
	\mathcal{A}_0:&=\{V_1^0, V_2^0, \ldots, V_q^0\},\\
	\mathcal{A}_j:&=\{V^j_1, V^j_2, \ldots, V^j_{r-1}\}, \qquad 1\leq j\leq p.
	\end{align*}
	Note that $ \mathcal{A}_0 $ is empty, whenever $ q=0 $. Hereafter, we handle the problem in the two cases based on parity of $ p $. \\
	
	\textbf{Case  1.} $ p=2k $ for some integer $ k $.\\
	Let $ E(\mathcal{H}^*)=\bigcup_{j=1}^{k}E_j $, where each
	$ E_j $ is a collection of edges defined as follows. Note that by $ x\in \mathcal{A}_h $ we mean that $ x $ is in a $ t $-tuple of $ \mathcal{A}_h $.
	\begin{align*}
	E_j:=\bigg\{ (x, y_1, y_2,\ldots, y_{r-1}) \big|\  &x\in \mathcal{A}_{k+j},\  (y_1,y_2,\ldots, y_{r-1})\in V^s_{1}\times V^s_{2}\times \cdots\times V^s_{r-1},\\
	& k-j+1\leq s\leq 2k,\ s\neq k+j
	\bigg\}.
	\end{align*}	
	Now, we show that $ \mathcal{A}_i $'s have different degrees in $ \mathcal{H}^* $. Since every edge in $ \mathcal{H} $ has size $ r $, it contains two vertices of different $ \mathcal{A}_i $'s.  Therefore, every edge in $ \mathcal{H} $ contains two vertices with different degrees in $ \mathcal{H}^* $.
	Let $ z\in \mathcal{A}_{j} $ and $ z'\in \mathcal{A}_{k+j'} $, for some $ 1\leq j,j'\leq k $. Then
	\begin{align}
	deg_{\mathcal{H}^*}(z)&=(r-1)jt^{r-1},\label{eq:dAj}\\
	deg_{\mathcal{H}^*}(z')&=\big[
	2k-(
	k-j'+1
	)
	\big]t^{r-1}
	+
	(k-1)
	(r-1)t^{r-1}\nonumber\\
	&=\big(kr-r+j'\big)t^{r-1}.\label{eq:dAj'}
	\end{align}
	If $ z\in \mathcal{A}_k $ and $ z'\in \mathcal{A}_{k+1} $, then the assumptions
	$ n>2(r-1)^2 $ and $ \lfloor\frac{n}{r-1}=2k\rfloor $ conclude that
	\begin{align}\label{eq:dA nesf}
	deg_{\mathcal{H}^*}(z)=(r-1)kt^{r-1}<(kr-r+1)t^{r-1}=deg_{\mathcal{H}^*}(z').
	\end{align}
	For every $ z''\in \mathcal{A}_0 $, $ deg_{\mathcal{H}^*}(z'')=0 $.  On the other hand, by relations
	\eqref{eq:dAj}, \eqref{eq:dAj'} and \eqref{eq:dA nesf}, it is clear that
	$ deg_{\mathcal{H}^*}(u) $ is monotone increasing in $ j $, for every arbitrary vertex $ u\in \mathcal{A}_j $,  $ 0\leq j \leq p $.  It completes the proof when $ p $ is even.\\

	\textbf{Case  2.} $ p=2k+1 $ for some integer $ k $.\\
	The proof is similar to what mentioned above.
	In this case, $ E(\mathcal{H}^*)=\bigcup_{j=1}^{k+1}E_j $, where
	\begin{align*}
	E_j:=\bigg\{ (x, y_1, y_2,\ldots, y_{r-1}) \big|\  &x\in \mathcal{A}_{k+j},\  (y_1,y_2,\ldots, y_{r-1})\in V^s_{1}\times V^s_{2}\times \cdots\times V^s_{r-1},\\
	& k-j+2\leq s\leq 2k+1,\ s\neq k+j
	\bigg\}.
	\end{align*}
	Let $ z\in \mathcal{A}_{j} $ and $ z'\in \mathcal{A}_{k+j'} $, for some $ 1\leq j\leq k $ and $ 1\leq j'\leq k+1 $. Then
	\begin{align}
	deg_{\mathcal{H}^*}(z)&=(r-1)jt^{r-1},\label{eq:dAj1}\\
	deg_{\mathcal{H}^*}(z')&=
	\big[2k+1-(
	k-j'+2
	)\big]t^{r-1}
	+
	k
	(r-1)t^{r-1}\nonumber\\
	&=\big(kr+j'-1\big)t^{r-1}.\label{eq:dAj'1}
	\end{align}
	For $ z\in \mathcal{A}_k $ and $ z'\in \mathcal{A}_{k+1} $, we have
	\begin{align}\label{eq:dA nesf1}
	deg_{\mathcal{H}^*}(z)=(r-1)kt^{r-1}<(kr+1-1)t^{r-1}=deg_{\mathcal{H}^*}(z').
	\end{align}
	By relations
	\eqref{eq:dAj1}, \eqref{eq:dAj'1} and \eqref{eq:dA nesf1},
	for every arbitrary vertex $ u\in \mathcal{A}_j $, $ deg_{\mathcal{H}^*}(u) $ is monotone increasing in $ j $, $ 0\leq j \leq 2k+1 $, that completes the proof for this case, as well.
\end{proof}

\begin{cor}
	Let $ n>2(r-1)^2 $ and $r\geq 3$. Then, $ \chi^{e}(\mathcal{K}_n^r(t))=\chi^{ve}(\mathcal{K}_n^r(t))= \chi^{ven}(\mathcal{K}_n^r(t))=2. $ Also, for $n\geq r+1$ and $r\geq 3$,
	$ \chi^{e}(\mathcal{K}_n^r)= \chi^{ve}(\mathcal{K}_n^r)=\chi^{ven}(\mathcal{K}_n^r)=2. $
\end{cor}
\begin{proof}
	Regularity of hypergraph $ \mathcal{K}_n^r(t) $ implies that $ \chi^{e}(\mathcal{K}_n^r(t))$ and $ \chi^{ve}(\mathcal{K}_n^r(t))$ are strictly greater than one. Also, $ \chi^{ven}(\mathcal{K}_n^r(t))> 1$, since $ |N_{\mathcal{K}_n^r(t)}(u)|=|N_{\mathcal{K}_n^r(t)}(v)|$, for every $ u, v \in V({\mathcal{K}_n^r(t)}) $.
	
	Let  $ w: E(\mathcal{K}_n^r(t))\rightarrow \{1,2\} $ be the edge weighting which yields a proper vertex coloring. This function exists by Theorem \ref{complete}. It can be generalized to a total weighting function $ w'$ by setting  $w'(v)=1$ for every vertex $v$, and $w'(e)=w(e)$, for all $ e\in E({\mathcal{K}_n^r(t)}) $.
	Hence,  $ \chi^{ven}({\mathcal{K}_n^r(t)}),\chi^{ve}({\mathcal{K}_n^r(t)})\leq\chi^{e}({\mathcal{K}_n^r(t)}) $. By Theorem \ref{complete}, $\chi^{ven}({\mathcal{K}_n^r(t)})=\chi^{ve}({\mathcal{K}_n^r(t)})=2 $.

	However, Theorem \ref{complete} yields $ \chi^{e}(\mathcal{K}_n^r)=2 $ for $ n>2(r-1)^2 $, in what follows we bring an argument which shows by a combination of this theorem and the previous argument we have $ \chi^{ve}({\mathcal{K}_n^r})=\chi^{ven}({\mathcal{K}_n^r})=2$, for $ n\geq r+1 $ and $ r\geq 3 $.
	
	To see this, we construct a hypergraph $ {\cal H}^*\subseteq{\mathcal{K}_n^r} $ such that for every selection of $ r $ vertices from $ V({\mathcal{K}_n^r}) $, always there exist $ i $ and $ j $, $ 1\leq i,j\leq r $, where $ \deg_{{\cal H}^*}(v_i)\neq\deg_{{\cal H}^*}(v_j) $. Now, we assign weight $ 2 $ to edge $ e $, if $ e\in E({\cal H}^*) $ and define weight $ e $ to be $ 1 $ for the remaining edges of $ \mathcal{K}_n^r $. This total weighting function yields the desired vertex coloring.
	
	The configuration of $ H^* $ is as follows. Assume that if $ n $ divided by $ r-1 $ equals $ p $ with reminder $ q $, i.e.  $ n=p(r-1)+q $. We define partition $ {\cal A}_0, {\cal A}_1, \ldots, {\cal A}_p $ of the vertex set $ V({\mathcal{K}_n^r}) $, where
	$ {\cal A}_i=\{x_1^i,x_2^i,\ldots, x_{r-1}^i\} $, $ 1\leq i\leq p $ and $ {\cal A}_0=\{x_1^0,x_2^0,\ldots, x_{q}^0\} $.
	First, suppose that $ p=1 $. By the assumption $ n\geq r+1 $, we have $ q\geq 2 $. In this case, let
	$ E({\cal H}^*)=\{{\cal A}_1\cup{x_k^0} \ | \ 1\leq k\leq q\} $. Note that if $ v\in {\cal A}_1 $, then $ \deg_{{\cal H}^*}(v)=q\geq 2 $ and for $ v\in {\cal A}_0 $, $ \deg_{{\cal H}^*}(v)=1 $. Clearly, in this way there is no $ r $-tuple of vertices with a same degree in $ {\cal H}^* $.
	Hence, we may assume that $ p\geq2 $. The edge set of $ {\cal H}^* $ is defined by $ E({\cal H}^*)=\bigcup_{i=1}^{p-1}\big\{{\cal A}_i\cup{x_k^j} \ | \ i+1\leq j\leq p, 1\leq k\leq r-1\big\}  $. Therefore,
	$$ \deg_{{\cal H}^*}(v)=
	\left\{
	\begin{array}{ll}
	(i-1)+(r-1)(p-i) & \ \ \ v\in {\cal A}_i, 1\leq i\leq p \\
	0 & \ \ \ v\in {\cal A}_0.
	\end{array}
	\right.
	$$
	It is readily to check that $ \deg_{{\cal H}^*}(v)\neq\deg_{{\cal H}^*}(u) $ for $ u\in {\cal A}_j $ and $ v\in {\cal A}_{j'} $, $ j\neq j' $ which completes the proof.
\end{proof}

\subsection{Paths and Cycles}
Let $t,r$ and $ n $ be integers, where $1\leq t <r$. An $r$-uniform $t$-tight path of order $n$, denoted by   $ \mathcal{P}_{n,t}^{(r)}, $ is a hypergraph with vertex set $[n]=\{1,2,\ldots, n\}$ and edge set $\{e_1,e_2,\ldots ,e_{\ell}\}$, such that ${\ell}=\frac{n-t}{r-t}$ is an integer and for $1\leq i\leq \ell$, $e_i=\{ (i-1)(r-t)+
1, (i - 1)(r - t) + 2, \dots , (i - 1)(r - t) + r\}$. Note that the edges are intervals
of length $r$ in $[n]$ and consecutive edges intersect in exactly $t$ vertices. Whenever $t=1$,  $ \mathcal{P}_{n,1}^{(r)} $
commonly referred as a loose path. Also, a $ \mathcal{P}_{n,r-1}^{(r)} $  is called a tight path.
An $r$-uniform $t$-tight cycle of order $n$, denoted by   $ \mathcal{C}_{n,t}^{(r)} $, is a hypergraph with vertex set $[n]$ and edge set $\{e_1,e_2,\ldots ,e_{\ell}\}$, where ${\ell}=\frac{n}{r-t}$ is an integer and for $1\leq i\leq \ell$, $e_i =\{ (i-1)(r-t)+
1, (i - 1)(r - t) + 2, \dots , (i - 1)(r - t) + r\}$. Note that here the index of $e_i$'s is considered on module $n$ and every two consecutive edges have exactly $t$ vertices in common. Cycles $ \mathcal{C}_{n,1}^{(r)} $ and  $ \mathcal{C}_{n,r-1}^{(r)} $  are called  loose  and tight cycles, respectively. The graphs path, $P_n$, and cycle, $C_n$,  are special cases of $ \mathcal{P}_{n,t}^{(r)} $ and $ \mathcal{C}_{n,t}^{(r)} $, whenever $r=2$ and $t=1$. Hence, a restriction of every result for hypergraphs concludes an statement for graphs.
By notation $e=\{x_1,x_2,\ldots, x_r\}$ for an edge $e$, we mean that there is a total ordering on  $x_i$'s, i.e.  $x_i<x_{i+1}$ for every $1\leq i \leq r-1$.

In the two following theorems we characterize $\chi^{e}(\mathcal{P}_{n,t}^{(r)})$ and $\chi^{e}(\mathcal{C}_{n,t}^{(r)})$  for every $n$, $r$ and $t$.

\begin{theorem}\label{path}
	Let $\mathcal{P}=\mathcal{P}_{n,t}^{(r)} $ be an $ r$-uniform $ t$-tight path of order $n$ and  length $ \ell>2 $. Then,
	\begin{equation*}
	\chi^{e}(\mathcal{P})=
	\left\{
	\begin{array}{cl}
	2 & \ \ \ t=\frac{r}{2}\ or\  t>\frac{r}{2},\  \text{r-t}\  divides \  \text{r}\ and\  \ell \geq 2\frac{r}{r-t}-1,  \\
	1 & \ \ \ \text{otherwise.}
	\end{array}
	\right.
	\end{equation*}	
\end{theorem}
\begin{proof}
	Let $ \cal P $ be the path on edges $ e_1, e_2, \ldots, e_{\ell} $, where $ \ell>2 $.
	Assume that $ t<\frac{r}{2} $ and  $ e_i=\{x_1,\ldots, x_r\} $  is an arbitrary edge. If $i\neq \ell$, then $ 1=\deg_{\cal P}(x_{r-t})<\deg_{\cal P}(x_r)=2 $. Also, for edge $ e_\ell $, we have $ 1=\deg_{\cal P}(x_{t+1})<\deg_{\cal P}(x_1)=2 $. Therefore, every edge of $ \cal P $ contains two vertices of different degrees and $ \chi^{e}({\cal P})=1 $.
	
	Let $ t=\frac{r}{2} $. Since degree of all the vertices in $ \cup^{\ell-1}_{i=2} e_i $ is two, $ \chi^e({\cal P})>1.$
	Now, suppose that $ w:E(P)\rightarrow\{1,2\} $ is an edge weighting which induces a proper vertex coloring.
	Clearly, $ w(e_i)\neq w(e_{i+2}) $, since otherwise, no matter what is $ w(e_{i+1}) $, all the vertices of $ e_{i+1} $ receive a same color.
	Hence, every weight assignment to an edge forces weight of alternative edges on the path. Thus, it is enough to weight two consecutive edges on the path.
	This fact implies that the only possible patterns to weight edges of a path are $ 1122\cdots $, $ 1221\cdots $, $ 2112\cdots $ and $ 2211\cdots $. Clearly, these weighting patterns yield a proper vertex coloring for $ \cal P $.  Consequently, $ \chi^e({\cal P})=2.$
	
	Now, let $ t>\frac{r}{2} $ and set $ k=\lfloor\frac{r}{r-t}\rfloor $. First assume that $ k=\frac{r}{r-t} $ is an integer and  $ \ell\geq 2k-1 $. In this case, for every vertex $ x\in e_{i} $, $k\leq i\leq \ell-k+1 $,  we have $ \deg_{\cal P}(x)=k $, thus  $  \chi^e({\cal P})>1  $. On the other hand, the following edge weighting function shows $  \chi^e({\cal P})=2  $. For $ j\geq0 $, let
	\begin{equation*}
	w(e_{jk+1})=w(e_{jk+2})=\cdots=w(e_{jk+k})=
	\left\{
	\begin{array}{cl}
	1 & \ \ \ j \ \text{is even,}  \\
	2 & \ \ \ j \ \text{is odd.}
	\end{array}
	\right.
	\end{equation*}
	Now, we suppose that $\frac{r}{r-t} $ is not an integer and $ \ell\geq 2k-1 $. Thus, $ r=k(r-t)+b_0 $, where $ b_0>0 $.
	It is easy to check that if $ 1\leq i\leq k $ or $ \ell-k+1\leq i\leq\ell $, then $ e_i $ contains two vertices with different degrees.
	On the other hand, for $ e_i=\{x_1, x_2, \ldots, x_r\} $, $ k+1\leq i\leq \ell-k  $,  we have $ k=\deg_{\cal P}(x_{r-b_0})<\deg_{\cal P}(x_r)=k+1 $.
	It concludes that
	$  \chi^e({\cal P})=1  $. Finally, for the case $ \ell<2k -1 $,  every edge of $\mathcal{P}$ contains two vertices with distinct degrees and hence, $ \chi^e({\cal P})=1  $.
\end{proof}
The following statement is an straightforward conclusion of Theorem \ref{path}.
\begin{cor}
	For every $ r$-uniform loose path $ \cal P $, we have $  \chi^e({\cal P})=1 $.
	Also, let $ \cal P $ be the $ r$-uniform tight path of length $ \ell $. If $ \ell\geq 2r-1 $, then $  \chi^e({\cal P})=2 $. Otherwise $  \chi^e({\cal P})=1 $.
\end{cor}

\begin{theorem}\label{cycle}
	Let $\mathcal{C}=\mathcal{C}_{n,t}^{(r)} $ be an $ r$-uniform $ t$-tight cycle of order $n$ and  length $ \ell>2 $. Then,
	\begin{equation*}
	\chi^{e}(\mathcal{C})=
	\left\{
	\begin{array}{cl}
	3 & \ \ \ t=\frac{r}{2}\ and\  \ell\overset{4}{\not\equiv}0,\\
	2 & \ \ \ t=\frac{r}{2}\ and\  \ell\overset{4}{\equiv}0\  or\  t>\frac{r}{2}\ and\  \text{r-t}\  divides \  \text{r},\\
	1 & \ \ \ \text{otherwise.}
	\end{array}
	\right.
	\end{equation*}
\end{theorem}
\begin{proof}
	Assume that $e_1,e_2,\ldots, e_{\ell}$ are the ordered edges of  $\mathcal{C}$.
	Let $ t<\frac{r}{2} $ and $ e=\{x_1,\ldots, x_r\} $  be an arbitrary edge. One can see that $ 2=\deg_{{\cal C}}(x_{1})>\deg_{{\cal C}}(x_{t+1})=1 $ and hence, $ \chi^e({{\cal C}})=~1 $.
	
	If  $ t=\frac{r}{2} $, then  $  {\cal C}$ is a $ 2$-regular hypergraph. Therefore, $ \chi^e({\cal C})>1 $.
	Let $ w:E(C)\rightarrow\{1,2\} $ be an edge weighting which induces a proper vertex coloring. Then, by the argument presented in the proof of Theorem \ref{path}, the only possible patterns to weight edges of  $\mathcal{C}$ are $ 1122\cdots $, $ 1221\cdots $, $ 2112\cdots $ and $ 2211\cdots $.
	If $ \ell\overset{4}{\equiv}0 $, then the first pattern yields a proper vertex coloring. One may readily check that no matter which pattern is applied, a monochromatic edge will be constructed when $ \ell\overset{4}{\not\equiv}0 $.
	Now, we show that $ \chi^{e}({\cal C})=3 $.
	We use the pattern $ 1122\cdots $ to weight the edges $ e_1,\ldots, e_{\ell-1} $  and put $ w(e_\ell)=3. $
	If $ \ell\overset{4}{\equiv}1 $, then we are done. Otherwise,
	it is enough to set $ w(e_{\ell-1}) = 3 $, as well.
	
	Now, suppose that  $ t>\frac{r}{2} $.
	First, let  $ k=\frac{r}{r-t} $, be an integer.
	In this case,
	$ {\cal C}$ is a $ k$-regular hypergraph.
	Thus,  $  \chi^e({\cal C})>1  $. Now, we prove that $  \chi^e({\cal C})=2  $. The assumption $ t>\frac{r}{2} $ implies that $ k\geq3 $.
	If $ \ell \leq 2k$, we define $w(e_1)=w(e_2)=2$ and $w(e_i)=1$ for every $i$, $i \neq 1,2$. Clearly, this weighting function yields a proper vertex coloring in $\mathcal{C}$.
	Otherwise, we may assume that $ \ell > 2k$.  Let $ {\cal P}^* $ be the longest non-spaning  $ (r-t)$-tight path in $ {\cal C}$ starting with edge $e_1$.
	Since
	$ r-t<\frac{r}{2} $, every $r$-consecutive vertices of  $ {\cal C}$ contains   two vertices with two different degrees $0$ and $1$  or $1$ and $2$ in $ {\cal P}^* $. Hence, the following weight function $ w $ causes a proper vertex coloring.
	\begin{equation*}
	w(e)=
	\left\{
	\begin{array}{cl}
	2 &e \in  E({\cal P}^*), \\
	1 & \text{otherwise.}
	\end{array}
	\right.
	\end{equation*}
	It is easy to see that every edge in $\mathcal{C}$, contains two vertices with different weights $k$ and $k+1$ or $k+1$ and $k+2$.
	Therefore,
	$  \chi^e({\cal C})=2  $.
	
	Finally, let $t> \frac{r}{2}$ and $ r=k(r-t)+b_0 $, where $ 0<b_0<r-t $. If  $ e_i=\{x_1,\ldots, x_r\} $  is an arbitrary edge, then $ k=\deg_{\cal C}(x_{r-b_0})<\deg_{\cal C}(x_{r})=k+1 $. Hence, $ \chi^e({\cal C})=1 $.
\end{proof}

\subsection{Theta hypergraphs}
For integers $ r,s\geq 3 $ and $t$,  let $p_1,p_2, \ldots, p_s$ be vertex disjoint $r$-uniform  $ t$-tight paths of lengths $ \ell_1, \ell_2, \ldots, \ell_s$, respectively.
The $ r$-uniform \textit{theta hypergraph}  $ \Theta_t^r(l_1,\ldots, l_s) $ is defined as  follows. First, unify all the first $ t $ vertices of $ p_i $'s and call them $ x_1,\ldots, x_{t} $. Then, identify  all the last $ t$ vertices of $ p_i $'s and nominate them $ y_1,\ldots, y_t$.  Note that, if $s=2$, then $ \Theta_t^r(l_1, l_2) $ is the $r$-uniform  $t$-tight cycle of length $ \ell_1+\ell_2$.
In the following, we classify theta hypergraphs based on the value of  parameter $ \chi^{e}$.

\begin{theorem}\label{teta}
	Let $r,s \geq 3$ and $t$ be integers and $ H_{\Theta}=\Theta_t^r(l_1,\ldots, l_s) $. Then,
	\begin{itemize}
		\item[i.] $ \chi^{e}(\mathcal{H}_{\Theta})=1$ when $t<\frac{r}{2}$.
		\item[ii.] For $t>\frac{r}{2}$, if $r-t$ divides $r$ and $\ell _i > 2(k-1)$ for some $1\leq i \leq s$, then
		$ \chi^{e}(\mathcal{H}_{\Theta})=2$. Otherwise, $ \chi^{e}(\mathcal{H}_{\Theta})=1$.
		
		\item[iii.]  If $t=\frac{r}{2}$, then
		\begin{equation*}
		\chi^{e}(\mathcal{H}_{\Theta})=
		\left\{
		\begin{array}{cl}
		3 & \ \ \  \ell_1=1\ \text{and}\  \ell_i\overset{4}{\equiv}1,\ \    2\leq i\leq s,\\
		1 & \ \ \   \ell_i=2,\ \  1\leq i\leq s,\\
		2 & \ \ \ \text{otherwise.}
		\end{array}
		\right.
		\end{equation*}
	\end{itemize}
\end{theorem}

\begin{proof}  Without loss of any generality, suppose that $\ell_1 \leq \ell_2 \leq \cdots \leq \ell_s$.\\
	\begin{itemize}
		\item[i.] Since every edge of $\mathcal{H}_{\Theta}$ contains two vertices with  different degrees, we deduce that $ \chi^{e}(\mathcal{H}_{\Theta})=~1$.
		
		\item[ii.] First, let $k=\frac{r}{r-t}$ be an integer. If $\ell _i \leq 2(k-1)$ for every $1\leq i \leq s$, then  every edge $e$ of $\mathcal{H}_{\Theta}$ has nonempty intersection with  set $\{x_1,\ldots, x_{t}, y_1,\ldots, y_{t}\}$. Also, $e$ contains two vertices with different degrees. Hence, $ \chi^{e}(\mathcal{H}_{\Theta})=1$.
		Now, suppose that  $\ell _j > 2(k-1)$ for some $1\leq j \leq s$. Thus, $k^{\text{th}}$ edge of $p_j$ is $k$-regular i.e. all of its vertices have degree $k$ in $\mathcal{H}_{\Theta}$. Therefore, $ \chi^{e}(\mathcal{H}_{\Theta})>1$. We show that $ \chi^{e}(\mathcal{H}_{\Theta})=2$.
		Note that, for every $i,j$,  $1\leq i\neq j\leq s$, the hypergraph $p_i \cup p_j$ is isomorphic to a $t$-tight cycle of length $\ell_i + \ell_j$ in $\mathcal{H}_{\Theta}$.
		Let  $\mathcal{H}^*=\cup_{j=2}^{s} E(\mathcal{P}^*_j)$, where $\mathcal{P}^*_j$ is the longest non-spaning $(r-t)$-tight path starting with the first edge of $p_1$, passing $p_1$ and then going through $p_j$ in cycle $p_1 \cup p_j$, whenever $\ell_j > 2(k-1)$.
		Since
		$ r-t<\frac{r}{2} $, every $r$-consecutive vertices of  cycle $ p_1 \cup p_j$ contains   two vertices with two different degrees $1$ and $2$, $1$ and $s$ or $1$ and $0$ in $\mathcal{H}^*$. Hence, the following weight function $ w $ yields a proper vertex coloring.
		\begin{equation*}
		w(e)=
		\left\{
		\begin{array}{cl}
		2 &e \in  E(\mathcal{H}^*), \\
		1 & \text{otherwise.}
		\end{array}
		\right.
		\end{equation*}
		It is easy to see that every edge in $\mathcal{H}_{\Theta}$, contains two vertices with different weights.
		Therefore,
		$\chi^e(\mathcal{H}_{\Theta})=2  $.
		
		Now, let $t>\frac{r}{2}$ and $r=k(r-t)+b_0$, where $0< b_0 <r-t$. One can readily see that every edge of $\mathcal{H}_{\Theta}$  contains two vertices of  different degrees. Therefore, $ \chi^{e}(\mathcal{H}_{\Theta})=1$.
				
		\item[iii.]  First, let $ \ell_1=1 $. If $ \ell_i\overset{4}{\equiv}1 $, for $ 2\leq i\leq s $,
		Clearly, $ \chi^{e}(\mathcal{H}_{\Theta})>1 $. To see $ \chi^{e}(\mathcal{H}_{\Theta})=3 $, suppose to the contrary that $ \chi^{e}(\mathcal{H}_{\Theta})=2 $ and $ w $ is the corresponding $ 2$-edge weighting. By the argument presented within the proof of Theorem \ref{path}, the only possible patterns to weight the edges of $ p_i $'s are $ 1122\cdots $, $ 1221\cdots $, $ 2112\cdots $ and $ 2211\cdots $.
		Since $\ell_i\overset{4}{\equiv}1 $, for every path $ p_i $, $2\leq i \leq s$, always the first edge and the last edge receive a same weight. No matter weight of the edge $ \{x_1,\ldots,x_{\frac{r}{2}}, y_1,\ldots, y_{\frac{r}{2}}\} $ of $p_1$, this edge is monochromatic which is a contradiction. Hence, $ \chi^{e}(\mathcal{H}_{\Theta})>2 $.
		Now, consider the following $ 3$-edge weighting for $ \mathcal{H}_{\Theta} $.  Apply the pattern $ 1122\cdots $, for $p_i$'s, $   2\leq i\leq s $, and weight $2$ to the  edge $ \{x_1,\ldots,x_{\frac{r}{2}}, y_1,\ldots, y_{\frac{r}{2}}\} $ of $ p_1 $. Then, change  weight of the first edge in $ p_2 $ to $ 3 $. It is easy to check that this weighting induces a proper vertex coloring for $ \mathcal{H}_{\Theta} $. Therefore, $ \chi^{e}(\mathcal{H}_{\Theta})=3$.
		
		Now, suppose that $ \ell_1=1 $ and $ \ell_i\overset{4}{\not\equiv}1$, for some $ 2\leq i\leq s $. %
		Assign $2$ to  the edge $ \{x_1,\ldots,x_{\frac{r}{2}}, y_1,\ldots, y_{\frac{r}{2}}\} $ of $ p_1 $. Also, apply the pattern $ 1122\cdots $ for $p_i$'s, when $ \ell_i\overset{4}{\equiv}0$ or $1$ and the pattern $ 1221\cdots $ for $p_i$'s,  whenever $ \ell_i\overset{4}{\equiv}2$ or $3$.  This weight assignments implies a proper vertex coloring for $ \mathcal{H}_{\Theta} $.
		
		Now, let $\ell_i \geq 2$ for all  $1\leq i\leq s$. If  $\ell_i = 2$ for all  $i$, then  every edge in $\mathcal{H}_{\Theta}$ contains two vertices with different degrees $2$ and $s$. The assumption $s \geq 3$ implies that $ \chi^{e}(\mathcal{H}_{\Theta})=1 $. Otherwise,  $ \ell_i \geq 3$ for some  $i$.  We apply the pattern  $ 2112\cdots $ for $p_i$, if $ \ell_i\overset{4}{\equiv}0$, the pattern $ 2211\cdots $, when  $ \ell_i\overset{4}{\equiv}1$ or $2$ and the pattern  $ 1122\cdots ,$  whenever $ \ell_i\overset{4}{\equiv}3$. This weighting yields a proper vertex coloring unless the case $s=3$, $ \ell_1,\ell_2\overset{4}{\equiv}3$ and $ \ell_3\overset{4}{\equiv}1$ or $2$. In this case, it is enough to change the pattern $p_3$ to $2112\cdots$.
	\end{itemize}
	\vspace*{-.5cm}
\end{proof}

\subsection{Geometric planes}
A projective plane of order $q$  is an ordered pair  $(X,\mathcal{L})$  in which  $X$ is  a set of $q^2+q+1$ points and $\mathcal{L}$  is a subset of power set of $  X $, where  $|\mathcal{L}|=q^2+q+1$ and satisfies the following conditions. Every member of $\mathcal{L}$ is a $(q+1)$-subset of $X$  called line, any point lies on $q + 1$ lines, every two points lie on a unique line and every two lines intersect in exactly one point.

An affine plane of order $q$ is an ordered pair  $(X,\mathcal{L})$  in which $X$ is  a set of   $q^2$ points  and $\mathcal{L}$  is a set of $q^2 + q$ lines, where
every line has $q$ points, any two points lie on a unique line and any point lies on $q + 1$ lines.
A parallel class in a geometric plane  is a  set of disjoint lines  partitioning  points set $X$.
It is a well known fact  that the set of lines of an affine plane of order $ q $ is partitioned into  $ q+1 $ parallel classes.

It is noteworthy that an affine plane of order $q$ exists if and only if  there is a projective plane of order $q$.
To see this, let $(X,\mathcal{L})$ be a projective plane of order
$q$. Now, consider  $(X',\mathcal{L}')$, where $X' =
X \setminus l_0$, $\mathcal{L}'= \{l \setminus {l}_0\ |\ l \in \mathcal{L}, l \neq l_0 \}$ and   $l_0$ is a fixed line in $ \mathcal{L}$. It is easy to verify that structure $(X',\mathcal{L}')$ is an affine plane of order
$q$.
Conversely, assume that $(X,\mathcal{L})$ is an affine plane of order
$q$ with parallel classes $ \Pi_1,\ldots,\Pi_{q+1} $. We construct a projective plane $(X',\mathcal{L}')$ of order $q$  as follows. Add a new line $l_0=\{\infty_{1}, \infty_{2}, \ldots, \infty_{q+1}\}$ with $ q+1 $ new points and let $X'= X \cup l_0 $. Then, set $\mathcal{L}'=l_0\cup\big(\bigcup_{l\in \Pi_i} (l\cup \{ \infty_i \})\big) $. The structure $(X',\mathcal{L}')$  is a projective plane of order $q$.

In the following theorem, we determine the parameter $\chi^e(\mathcal{H})$ when $\mathcal{H}$ is an affine plane or a projective plane.

\begin{theorem}\label{plane}
	Let  $ \cal H $ be either an affine plane or a projective plane of order $ q $. Then $ \chi^e(\mathcal{H})=2 $.
\end{theorem}
\begin{proof}
	First assume that $ \cal H $ is an affine plane of order $ q $ with parallel classes $ \Pi_1,\ldots,\Pi_{q+1} $.
	By definition, for every line $ \ell\not\in \Pi_{q+1} $, $ |\ell\cap \ell_0|=1 $, where $ \ell_0=\{x_1,x_2,\ldots, x_q\} $ is a fixed line in $ \Pi_{q+1} $.
	Let $ \ell_i^j $ be the unique line in class $ \Pi_j $ containing $ x_i $, $ 1\leq i,j\leq q $.
	Define weighting function $ w:E({\cal H})\rightarrow\{1,2\} $  in which lines $ \ell_i^i $,  $ 1\leq i\leq q-1 $, take weight $ 2 $ and  all the remaining lines receive weight $ 1 $. 
	We claim that $ w $ induces a proper vertex coloring.  For every $ \ell\in \Pi_{q}\cup\Pi_{q+1} $, there are $ q-1 $ lines of weight $ 2 $ each of them crossing $ \ell$ in exactly one point. Therefore, there are two elements in $ \ell$, say $ z,z' $, so that
	$ \sigma^{e}(z)=q+1<q+2 \leq \sigma ^{e}(z') $.
	Now, take an arbitrary line $ \ell'\in \Pi_{i} $, $ 1\leq i\leq q-1 $. There are $ q-2 $ lines of weight $ 2 $ crossing $ \ell' $ which implies existence of two points with different colors in $ \ell' $.
	
	Let $ \mathcal {H}' $ be the projective plane of order $ q $ obtaining from $ \cal H $ by extending line set by the line $ \ell^*=\{\infty_{1}, \ldots, \infty_{q+1}\} $ and every $ \ell\in \Pi_i $, $ 1\leq i\leq q+1 $, by $ \ell\cup{\infty_{i}} $.
	Now,  we define weighting function $ w' $ for $ \cal H' $, using weighting function $ w $ as introduced above for affine plane $ \cal H $. Set $ w'(\ell^*)=1 $ and for every other line $ \ell'\in{\cal B(H')}\setminus \ell^* $, let the weight $ w'(\ell') $ be the same as its corresponding line in the affine plane. Then, $ \sigma ^{e}(\infty_{q+1})=q+1< q+2=\sigma ^{e}(\infty_{1}) $ that completes the proof.
\end{proof}

\section{Concluding Remark}

In Theorem \ref{g1}, some complete $n$-partite  graphs with parts of arbitrary sizes are considered and  the parameter $\chi^{e}$ is obtained    for them.  Theorem \ref{g2} determines $\chi^{ven}$ for complete $3$-partite graphs with parts of arbitrary sizes. Now, one may raise the following question.
\begin{problem}
	Let $ G $ be the complete $n$-partite  graph with parts of arbitrary sizes. What are the values of  $ \chi^{e}(G) $ and $ \chi^{ven}(G) $?
\end{problem}

It was proved that if $ G $ is a $ 3$-colorable graph, then $ \chi^{e}(G)\leq 3$, \cite{KLT}. Consequently, the 1-2-3-conjecture is true for bipartite graphs.  The only known bipartite graphs for which $ \chi^{e}$ reaches $3$ was an special subclass of theta graphs.
 This observation motivated  Khatirinejad et al. to  ask the following question \cite{Khatir}.
\begin{problem} \em{\cite{Khatir}}\label{khatir}
	Is it true that for every bipartite graph except an special subclass of theta graphs, we have $ \chi^{e}(G)\leq 2 $?
\end{problem}
Davoodi and Omoomi answered the problem in negative by extending  this family to bipartite generalized polygon trees as a generalization of theta graphs (Theorem 4.4. in \cite{DO}).
Note that although there are many sufficient conditions for $ \chi^{e}(G)\leq 2 $, the problem of classifying graphs with $ \chi^{e}(G)\leq 2 $ is still an open problem, even for graphs where correctness of 1-2-3-conjecture was confirmed for them.

A similar observation can be considered for hypergraphs, as well.
Clearly, for a hypergraph $ \mathcal{H} $, $ \chi^e(\mathcal{H})=1 $ if and only if every edge of $\mathcal{H}$ contains two vertices with different degrees.
A pair of vertices $ u $ and $ v $ are called \textit{twins}, whenever the set
of edges containing $ u $ is the same as the set of edges containing $ v $.
A hypergraph is called \textit{twin-free} if it contains no twins. For a graph $ G $, let $ \mathcal{H}^r(G) $ be the family of $r$-uniform  hypergraphs where every hypergraph  $\mathcal{H}$ in  $\mathcal{H}^r(G) $ is constructed  as follow. Every vertex  $v \in V( G) $ is replaced with a set $S_v$ of new vertices such that two  vertices $u$ and $v$ are adjacent in $ G $ if and only if the set $S_u \cup S_v$ is an edge of size $r$ in $\mathcal{H}$.
For example, a cycle of length $4k+2$ can be transformed  to an $r$-uniform hypergraph by setting $ |S_v|=r-1 $ for every other vertex on the cycle \cite{BDFL}.
This procedure can be generalized to every bipartite graph. It is sufficient to consider a breadth first search(BFS) tree of $G$ and then replace the vertices in odd levels with $ S_v $'s of size $p$ and   the vertices in even levels with $ S_v $'s of cardinality $r-p$. Also, it is clear that a simple way to reach $ \cal H $, for an even $ r $,  is to transform every vertex $ v $ to a set $ S_v $ of size $ \frac{r}{2} $.
Note that $|E(G)|=|E( \mathcal{H})|$ and $ \chi^{e}(G)=  \chi^{e}( \mathcal{H})$ for any $\mathcal{H}\in \mathcal{H}^r(G) $. Also, every member of $ \mathcal{H}^r(G) $ has twins.
Define
$$ \mathfrak{A}^r=\{\mathcal{H} : \mathcal{H}  \text{ is an } r\text{-uniform hypergraph and}  \chi^{e}(\mathcal{H})\leq 2 \}, $$
$$  \mathfrak{B}^r=\{\mathcal{H} : \mathcal{H}  \text{ is an } r\text{-uniform hypergraph and }  \chi^{e}(\mathcal{H})> 2 \}. $$
Bennett and Dudek proved that $  \mathfrak{B}^r \neq\emptyset $ (Theorem 1 (ii) in \cite{BDFL}). Moreover, we presented  examples of cycles and their generalizations to some theta hypergraphs $\mathcal{H}$ for  which $ \chi^e(\mathcal{H})=3 $ in Theorems \ref{cycle} and \ref{teta}.
In \cite{KKP-H} Kalkowski et al.  proposed the following conjecture.
\begin{conjecture}{\em\cite{KKP-H}}
	There is no twin-free hypergraph in  $  \mathfrak{B}^r$.
\end{conjecture}
It  can be seen that all known $r$-uniform hypergraphs in $\mathfrak{B}^r$ (obtained in  \cite{BDFL} and this paper),  belong to $\mathcal{H}^r(G) $ for some graph $G$ with $ \chi^{e}(G)=3$.
By Theorem \ref{thm:almost}, for an $ r $-uniform random hypergraph   $ \mathcal{H}= \mathcal{H}^{r}_{n,p}  $, asymptotically almost surely,  $ \chi^{e}(\mathcal{H})=1 $, if $ r\geq 4  $ and $ H\in  \mathfrak{A}^r $, whenever $ r=3 $.   Hence, motivated by Conjecture \ref{conj:3-HYPER} and Theorem \ref{thm:almost}, classifying  hypergraphs in $ \mathfrak{B}^r$ is another interesting problem. We conclude this section by proposing the following question.
\begin{problem}
	Is there any $ r$-uniform hypergraph $ {\cal H} \in {\mathfrak{B}}^r $, such that $ \mathcal{H}\notin  \mathcal{H}^r(G)$ for some graph $G$?
\end{problem}


\begin{thebibliography}{10}
	\bibitem{A1}
	L. Addario-Berry, K. Dalal and B. A. Reed, ``Degree constrained subgraphs",  \textit{Discrete Appl. Math.}, 156(7) (2008), 1168-1174.
	
	\bibitem{A2}
	L. Addario-Berry, K. Dalal, C. Mcdiarmid, B. A. Reed and A. Thomason, ``Vertex colouring
	edge-weightings", \textit{Combinatorica}, 27 (2007), 1-12.
	
	\bibitem{BDFL}
	P. Bennett, A. Dudek, A. Frieze  and L. Helenius, ``Weak and strong versions of the 1-2-3 conjecture for uniform hypergraphs",  \textit{Electron. J. Combin.}, 23(2) (2016),  2-46.
		
	\bibitem{B}
	O. Baudon, H. Hocquard, A. Marczyk,  M. Pil\'{s}niak, J. Przyby\l o and M. Wo\'{z}niak, ``On a total version of $1,2,3$ Conjecture", 2018. hal-01754080.
	
	\bibitem{CH}
	G. Chang, C. Lu, J. Wu and Q. Yu, ``Vertex-coloring
	edge-weightings of graphs", \textit{Taiwanese
		J. Math.}, 15(4) (2011), 1807-1813.
	
	\bibitem{DO}
	A. Davoodi and B. Omoomi, ``On the 1-2-3-conjecture", \textit{Discrete Math. Theor. Comput. Sci.}, 17(1) (2015), 67-78.

	\bibitem{K}
	M. Kalkowski, ``A note on 1, 2-conjecture",  Ph.D. Thesis (UAM Pozna\'{n}, 2009).
	
	\bibitem{KLT}
	M. Karo\'{n}ski, T. \L uczak and A. Thomason, ``Edge weights and vertex colours", \textit{J.
		Combin. Theory, Series B}, 91 (2004), 151-157.
	
	\bibitem{KKP-H}
	M. Kalkowski,  M. Karo\'nski and F. Pfender. ``The  1,2,3 conjecture  for  hypergraphs", \textit{J. Graph Theory} 85(3) (2017), 706-715.
	
	\bibitem{KKP}
	M. Kalkowski, M. Karo\'nski and F. Pfender, ``Vertex-coloring edge-weighting: toward
	the 1-2-3-conjecture", \textit{J. Combin. Theory, Series B}, 100 (2010), 347-349.
	
	\bibitem{Khatir}
	M. Khatirinejad, R. Naserasr, M. Newman, B. Seamone, and B. Stevens, ``Vertex-colouring edge weightings with two edge weights", Discrete Math. Theor. Comput. Sci. 14 (2012), no. 1, 1-20.
	
	\bibitem{LYZ}
	H. Lu, Q. Yu  and C-Q. Zhang,  ``Vertex-coloring 2-edge-weighting of graphs", \textit{European J. Combin.}, 32(1) (2011), 21-27.
	
	\bibitem{PW}
	J. Przyby\l o and M. Wo\'{z}niak, ``On a 1, 2 conjecture", \textit{Discrete Math. Theor. Comput.
		Sci.}, 12 (2010), 101-108.
	
	\bibitem{survey}
	B. Seamone, ``The 1-2-3 conjecture and related problems: a survey", \textit{arXiv preprint arXiv:1211.5122} (2012).
		
	\bibitem{WY}
	T. Wang and Q. Yu, ``On vertex-coloring $13$-edge-weighting", \textit{Front. Math. China}, 3(4) (2008), 581-587.
	
\end{thebibliography}
\end{document}